\documentclass[11pt]{amsart}
\usepackage{amsmath,amstext,amsthm,amssymb,amsxtra, amsthm,amscd,amssymb, units,fullpage}
\usepackage{setspace}
 \usepackage[margin=1in]{geometry}
\usepackage[colorlinks,citecolor=red,pagebackref,hypertexnames=false]{hyperref}

\title[On sumsets of subgroups in $\mathbb Z_p^*$]{A note on sumsets of subgroups in $\mathbb Z_p^*$.}

\author{Derrick Hart}
\address{ Department of Mathematics\\
          Kansas State University\\
          Manhattan, KS 66506}
\email{dnhart@math.ksu.edu}

 \newcommand{\comments}[1]{}

\begin{document}
\newtheorem{theorem}{Theorem}
\newtheorem{conjecture}[theorem]{Conjecture}
\newtheorem{corollary}[theorem]{Corollary}
\newtheorem{lemma}[theorem]{Lemma}
\theoremstyle{remark}
\newtheorem{remark}{Remark}[section]
\newtheorem{problem}{Open problem}

\maketitle

\setstretch{1.2}

\begin{abstract}
Let $A$ be a multiplicative subgroup of $\mathbb Z_p^*$. Define the $k$-fold sumset of $A$ to be $kA=\{x_1+\dots+x_k:x_i \in A,1\leq i\leq k\}$.  We show that $6A\supseteq \mathbb Z_p^*$ for $|A| > p^{\frac {11}{23} +\epsilon}$.  In addition, we extend a result of Shkredov to show that $|2A|\gg |A|^{\frac 85-\epsilon}$ for $|A|\ll p^{\frac 59}$.

\end{abstract}
\section{Introduction}
For subsets $A_1,\ldots,A_k$ of a group define $A_1+\dots+A_k=\{a_1+\dots+a_k: a_i\in A_i, 1\leq i\leq k\}$.  In the case that all the subsets are equal we will denote the $k$-fold sumset of $A$ by $kA=\{x_1+\dots+x_k:x_i \in A,1\leq i\leq k\}$.

Let $A$ be a multiplicative subgroup of $\mathbb Z_p^*$.  What is the smallest $\alpha>0$ such that $|A|\gg p^{\alpha}$ implies that $2A$ contains $\mathbb Z_p^*$?
\begin{conjecture}
Let $|A|> p^{\frac 12 +\epsilon}, \epsilon>0$ then $2A$ contains $\mathbb Z_p^*$. 
\end{conjecture}


It is relatively simple, using exponential sum bounds, to show that if $|A|>p^{\frac 34}$ then $2A\supseteq \mathbb Z_p^*$.  Surprisingly, no improvement in the exponent has been made.
An alternative approach would be to consider this conjecture from an inverse perspective.  Let $|A|>p^{\frac 12 +\epsilon}$; what is the smallest $k_0$ such that $k_0A$ contains $\mathbb Z_p^*$? A direct application of classical counting methods using standard exponential sum bounds does not seem to yield any answer to this question. For example, using the fact that $\max_{\lambda\neq 0}|\sum_{x \in A} e_p(x\lambda)|\leq \sqrt{p}$ one may show that if $|A|> p^{\frac 12+\frac 1{2k}}$ then $kA$ contains $\mathbb Z_p^*$.   

Using combinatorial methods Glibichuk \cite{glibichuk1} gave the first answer to this question showing that $8A\supseteq \mathbb Z_p^*$ for $|A|\geq 2 p^{\frac 12}$.       Using an improved exponential sum bound,  Schoen and Shkredov \cite[Theorem 2.6]{ss1} showed that $7A\supseteq \mathbb Z_p^*$ for $|A|> p^{\frac 12}$.  There was subsequent improvement to this result by Shkredov and Vyugin \cite{sv} followed by Schoen and Shkredov \cite{ss2}.   Recently, Shkredov \cite{s} has shown that $6A\supseteq \mathbb Z_p^*$ if $|A|> p^{\frac {55}{112}+\epsilon}=p^{.491\ldots+\epsilon}$.

 In this paper we elaborate on the methods in the above mentioned papers to show that $6A\supseteq \mathbb Z_p^*$ if $|A| >  p^{\frac {11}{23} +\epsilon}=p^{.478\ldots +\epsilon}$.   In addition, we extend a result of Shkredov (\cite{s}) to show that $|2A|\gg |A|^{\frac 85-\epsilon}$ for $|A|\ll p^{\frac 59}$.


\section{Statement of  Main Results}
 Let $A$ and $B$ be subsets of $\mathbb Z_p$.  Given a set $A$ we will denote the indicator function of $A$ by $A(\cdot)$.
Define the convolution of $A$ and $B$ by $(A\ast B)(z)=\sum_{x+y=z}A(x)B(y)=|A\cap(B+z)|.$

The additive energy between $A$ and $B$ be given by, 
\begin{align*}E(A,B)=&|\{(x,y,z,w)\in A\times B\times A\times B: x+y=z+w\}|\\=&\sum_z (A\ast B)^2(z)=\sum_z |A\cap(z-B)|^2\\
=&\sum_z (A\ast -A)(z)(B\ast -B)(z)=\sum_z |A_z||B_z|,\end{align*}
where here and throughout the paper we will let $C_z=C \cap (C+z)$ for any subset $C$ of $\mathbb Z_p$. In the case that $A=B$ we will write $E(A)=E(A,A)$.  Similarly, we will denote the $r$th additive energy of a subset $A$ by $E_r(A)=\sum_s |A_s|^{r}.$ 

One may also consider the additive energy in the frequency domain. Taking an exponential sum expansion, $E(A,B)=p^{-1}\sum_s \left|\sum_{x\in A} e_p(sx)\right|^2\left|\sum_{y\in A} e_p(sy)\right|^2,$ where $e_p(x)=e^{\frac {2\Pi ix}{p}}$. For a subset $A$ of $\mathbb Z_p$ we define $\Phi_A=\max_{\lambda \neq0} \left|\sum_{x\in A} e_p(\lambda x)\right|.$

Heath-Brown and Konyagin employed Stepanov's method in order to give a bound on the additive energy of multiplicative subgroups of $\mathbb Z_p^*$.
\begin{theorem}[\cite{hk}]\label{konyaginenergy}
Let $A$ be a multiplicative subgroup of $\mathbb Z_p^*$ with $|A|\ll p^{\frac 23}$. Then
$$E(A)\ll |A|^{\frac 52}.$$
\end{theorem}
In \cite{s} Shkredov gave the following combinatorial lemma.
\begin{lemma}[\cite{s}, Equation 1)]\label{ShSc}
Let $A$ be a finite subset of an abelian group. Then 
$$\sum_s \frac{|A_s|^2} {|A+ A_s|} \ll |A|^{-2}E_3(A).$$
\end{lemma}

Schoen and Shkredov (\cite{ss1}) gave an estimate for $E_3(A)$.
\begin{lemma}[\cite{ss1}, Lemma 3.3]\label{ShSc3}
Let $A$ be a multiplicative subgroup $A$ of $\mathbb Z_p^*$ with $|A|\ll p^{\frac 23}$. Then we have,
$$E_3(A)\ll|A|^3\log(|A|).$$
\end{lemma}

Combining Lemma \cite{s} and Lemma \cite{ss1} and noting that $|A+A_s|\leq|(2A)_s|$ gives the following lemma.
\begin{lemma}\label{ShSc2}
Let $A$ be a multiplicative subgroup $A$ of $\mathbb Z_p^*$ with $|A|\ll p^{\frac 23}$. Then we have,
$$\sum_{s} \frac{|A_s|^2}{|(2A)_s|}\ll|A|\log(|A|).$$
\end{lemma}

Shkredov used this inequality in \cite{s} to give the following estimate on the additive energy.
\begin{theorem}[\cite{s}, Theorem 30]\label{energy1shkredov} Let $A$ be a multiplicative subgroup of $\mathbb Z_p^*$ such that $|A|\ll p^{\frac 23}$. If $E(A)\ll |A|^{\frac 32}\sqrt{p}\log(|A|)$ then
$$E(A)\ll|A|^{\frac 43}|2A|^{\frac 23}\log(|A|).$$
\end{theorem}
In addition, using different methods he proved an energy estimate independent of the size of the sumset.
\begin{theorem}[\cite{s}, Theorem 34]\label{energy2shkredov} Let $A$ be a multiplicative subgroup of $\mathbb Z_p^*$ such that $|A|\ll p^{\frac 23}$. Then 
$$E(A)\ll \max\{|A|^{\frac {22}9}\log(|A|), |A|^3p^{-\frac 13}\log^{\frac 43}(|A|)\}.$$
\end{theorem}

Combining Theorem \ref{energy1shkredov} and Theorem \ref{energy2shkredov} and applying the trivial estimate $|2A|\geq|A|^4 E^{-1}(A)$ gives the following sumset estimate.
\begin{theorem}Let $A$ be a multiplicative subgroup of $\mathbb Z_p^*$ such that $|A|\ll p^{\frac 23}$.  Then
 \begin{equation*}|2A|\gg
 \begin{cases} |A|^{\frac 85}\log^{-\frac 35}(|A|),  &\text{if \quad  $|A|\ll p^{\frac 9{17}}$;} \\
 |A|^{\frac {14}9}\log^{-1}(|A|), &\text{if \quad  $|A|\ll p^{\frac 35}\log^{\frac 35}(|A|)$;}\\
|A|p^{\frac 13}\log^{-\frac 43}(|A|), &\text{if \quad  $|A|\gg p^{\frac 35}\log^{\frac 35}(|A|)$.}
\end{cases}
\end{equation*}
\end{theorem}

Here we give the following energy estimate.
\begin{theorem}\label{energyextension} Let $A$ be a multiplicative subgroup of $\mathbb Z_p^*$ such that $|A|\ll p^{\frac 23}$. Then 
$$E(A)\ll \max\{|A|^{\frac 43}|2A|^{\frac 23}\log^{\frac 12}(|A|), |A||2A|^2p^{-1}\log(|A|)\}.$$
\end{theorem}

This allows us to improve Shkredov's sumset result in some ranges.
\begin{theorem}Let $A$ be a multiplicative subgroup of $\mathbb Z_p^*$ such that $|A|\ll p^{\frac 23}$.  Then 
 \begin{equation*}|2A|\gg
 \begin{cases} |A|^{\frac 85}\log^{-\frac 3{10}}(|A|),  &\text{if \quad  $|A|\ll p^{\frac 59}\log^{-\frac 1{18}}(|A|)$;} \\
|A|p^{\frac 13} \log^{-\frac 1{3}}(|A|), &\text{if \quad  $|A|\gg p^{\frac 59}\log^{-\frac 1{18}}(|A|)$.}
\end{cases}
\end{equation*}
\end{theorem}

Using, Plancherel or orthogonality one can very quickly prove that for a multiplicative subgroups $A$, $\Phi_A\ll \sqrt p$ for $|A|\gg p^{\frac 12}$.  This is only non-trivial when $|A|>p^{\frac 12}$.
Shparlinski (\cite{shp}) improved this result in some ranges with the bound $\Phi_A\ll |A|^{\frac 7{12}}p^{\frac 16}$ for $p^{\frac 25}\ll |A|\ll p^{\frac 47}$.
Heath-Brown and Konyagin used the energy estimate of Theorem \ref{konyaginenergy} to obtain the following improvement. 
\begin{theorem}Let $A$ be a multiplicative subgroup. Then,
\begin{equation*}
\Phi_A\ll
 \begin{cases}\sqrt{p},  &\text{if \quad  $p^{\frac 23}\ll|A|\leq p$;} \\
p^{\frac 14}|A|^{-\frac 14}E^{\frac 14}(A)\ll p^{\frac 14}|A|^{\frac 38} , &\text{if \quad   $p^{\frac 12}\ll|A|\ll p^{\frac 23}$.}\\
p^{\frac 18}E^{\frac 14}(A)\ll p^{\frac 18}|A|^{\frac 58} , &\text{if \quad   $p^{\frac 13}\ll|A|\ll p^{\frac 12}$.}
\end{cases}
\end{equation*}
\end{theorem}

Using Shkredov's energy estimate, then one may improve this result in some ranges in the case that the sumset is small. Let $|A| \ll p^{\frac 12}$ then,
 $$ \Phi_A \ll p^{\frac 18} |A|^{\frac 1{3}}|2A|^{\frac 16}\log^{\frac 14}.$$

Using the same methods used to prove Lemma \ref{ShSc3} one may obtain $ E_{3/2}(A)\ll |A|^{\frac 94}$.  In the case that the sumset is small we are able to significantly improve this bound.
\begin{lemma}\label{E3/2energy}Let  A be a multiplicative subgroup with $|A|\ll p^{\frac 12}$. Then 
$$ E_{3/2}(A) \ll|A|^{\frac 12}|2A| \log^{\frac 74}{|A|}.$$ 
\end{lemma}

This Lemma allows us to obtain the following exponential sum bound which is an improvement of the result of Shkredov as long as $|2A|\ll |A|^{\frac 74}.$
\begin{lemma}\label{expA} Let $A$ be a multiplicative subgroup with $|A|\ll p^{\frac 12}$.Then
$$ \Phi_A \ll p^{\frac 18} |A|^{-\frac 1{8}}|2A|^{\frac 14}E^{\frac 18}(|A|) \log^{\frac 7{16}}(|A|).$$
In particular, applying Theorem \ref{energyextension} we have
$$ \Phi_A \ll p^{\frac 18} |A|^{\frac 1{24}}|2A|^{\frac 13} \log^{\frac 5{8}}(|A|).$$
\end{lemma}

With Lemma \ref{expA} in tow, we may now prove our main result.
\begin{theorem}\label{yay2} Let $A$ be a multiplicative subgroup of $\mathbb Z_p^*$ with $|A|\gg  p^{\frac {11}{23}}\log^{\frac{15}{23}}(|A|)$.
Then $$6A\supseteq \mathbb Z_p^*.$$
\end{theorem}
\begin{proof} Fix $a$ in $\mathbb Z_p^*$. We my assume that $|A|\ll p^{\frac 12}$ as the result is already known in the range $|A|\gg p^{\frac 12}$.

Let $N$ be the number of solutions to the equation,
\begin{equation*}
x_1+x_2+y_1+y_2=a y_3,
\end{equation*}
with $x_1,x_2 \in 2A$ and $y_1,y_2.y_3 \in A$. 

Taking an exponential sum expansion,
$$N=\frac{|2A|^2|A|^3}{p}+\frac{1}{p} \sum_{\lambda\neq 0} \left(\sum_{x \in 2A} e_p(\lambda x)\right)^2 \left(\sum_{y \in A} e_p(\lambda y)\right)^2\left(\sum_{z \in A} e_p(-\lambda za)\right),$$
which by Plancherel
implies that we have that $N>0$ as long as,
$|2A||A|^3>p\Phi^3_A$. 

Applying Theorem  gives the condition,
$$|2A| |A|^3\gg  p^{\frac {11}8} |2A||A|^{\frac 1{8}} \log^{\frac {15}{8}}(|A|),$$
which in turn gives the condition,
$$ |A|\gg p^{\frac {11}{23}}\log^{\frac{15}{23}}(|A|).$$
\end{proof}

\section{A Few Preliminary Lemmas}
We begin with a lemma of Shkredov and Vyugin \cite[Corollary 5.1]{sv} which is a generalization of a result of Heath-Brown and Konyagin \cite{hk}.  We say that a subset $S\neq \{0\}$ is $A$-invariant if $SA=\{sa:s \in S,a\in A\}=S$, that is $S$ is a union of cosets of $A$ and possibly $\{0\}$. 
\begin{lemma}{{\bf (Shkredov and Vyugin \cite[Corollary 5.1]{sv})}}\label{convolution} Let $A$ be a multiplicative subgroup of $\mathbb Z_p$ and $S_1, S_2, S_3$ be $A$-invariant sets such that 
$|S_1\setminus\{0\}||S_2\setminus\{0\}||S_3\setminus\{0\}|\ll \min\{|A|^5, p^3 |A|^{-1}\}$.  Then 
$$\sum_{z\in S_3} (S_1 \ast S_2)(z)\ll |A|^{-1/3}(|S_1||S_2||S_3|)^{2/3}.$$ 
\end{lemma}
\begin{remark}
The above lemma has been modified slightly from its original form in order to allow $S_1,S_2,S_3$ contain the zero element.  One may check that the additional terms in $\sum_{z\in S_3} (S_1 \ast S_2)(z)$ allowing $S_1$, $S_2,$ and to contain the zero element only affect the implied constant.
\end{remark}
We can now give slight generalizations of several results of Schoen and Shkredov (\cite{ss1},  \cite{ss2}).
\begin{lemma}\label{L3energy}Let $k\gg 1$ and $S_1,S_2$ be $A$-invariant sets and let $M$ be any $A$-invariant subset of the set $\{z:(S_1 \ast S_2)(z) \geq k\}.$ If
$|S_1||S_2||M||A|\ll \min \{|A|^6, p^3\}$ then for $r\geq 1,r\neq 3$,
$$ \sum_{z \in M}  (S_1 \ast S_2)^r(z) \ll |S_1|^{2}|S_2|^{2} |A|^{-1}k^{r-3},$$   
and
$$ \sum_{z\in M}  (S_1 \ast S_2)^3(z) \ll |S_1|^2|S_2|^2 |A|^{-1}\log(|S_1|^2|S_2|^2|A|^{-2}k^{-3}).$$ 

\end{lemma}
  
\begin{proof}
Let $l_i= (S_1 \ast S_2)(z_i),z_i\neq 0$ where $l_1\geq l_2\geq\dots$ are arranged in decreasing order. For each $z$ in the coset $aA=\{aa':a'\in A\}, a\in\mathbb Z_p$ note that $ (S_1 \ast S_2)(z)=(S_1 \ast S_2)(a)$. By the coset $a_iA$ we will mean the coset on which $l_i=(S_1 \ast S_2)(a_i)$.
Let $M$ be any $A$-invariant subset of the set $\{z:(S_1 \ast S_2)(z) \geq k\}$ and $M_i=\cup_{j=1}^i a_jA\subseteq M$. 
 From Lemma \ref{convolution} we have that 
$$l_i |A| i \leq\sum_{j=1}^{i}|A| l_j \leq \sum_{z\in M_i}(S_1 \ast S_2)(z)\ll  i^{2/3} |A|^{\frac 13}|S_1|^{\frac 23}|S_2|^{\frac 23},$$ as long as $i|A| |S_1||S_2|\ll |M| |S_1||S_2|\ll \min\{|A|^5, p^3 |A|^{-1}\}.$  
Now,
\begin{align*} \sum_{z\in M}  (S_1 \ast S_2)^r(z)\leq&  |A| \sum_{i\ll |S_1|^3|S_2|^3|A|^{-2}k^{-3}} l_i^r\\
\ll &|A| \sum_{i\ll |S_1|^2|S_2|^2|A|^{-2}k^{-3}} \left( i^{-\frac 13} |A|^{-\frac 23}|S_1|^{\frac 23}|S_2|^{\frac 23} \right)^r.
\end{align*}
\end{proof}

\comments{ In a similar way one may obtain the following bound.
\begin{lemma}\label{L3corollary} Let $S_1,S_2, S_3$ be $A$-invariant sets with $\max\{|S_1|,|S_2|, |S_3|\} \ll \min\{|A|^{3/2},p|A|^{-1/2}\}$.  Then
$$  \sum_z (S_1 \ast S_1)(z) (S_2\ast S_2)(z) (S_3 \ast S_3)(z)  \ll P |A|^{-1}|S_1|^{\frac 43}|S_2|^{\frac 43}|S_3|^{\frac 43},$$
where $P= \log^{\frac 13}(|S_1|^2|A|^{-1})\log^{\frac 13}(|S_2|^2|A|^{-1})\log^{\frac 13}(|S_3|^2|A|^{-1})$. 
\end{lemma}  }

\section{Additive Energy Bound: Proof of Theorem \ref{energyextension}}
We may assume that $E(A)\gg  \max\{|A|^{\frac 43}|2A|^{\frac 23}\log^{\frac 12}(|A|), |A||2A|^2p^{-1}\log(|A|)\}.$ Combining this with the energy estimate from Theorem \ref{konyaginenergy} we may also assume that 
$$|2A|\ll \max\{|A|^{\frac 74}\log^{-\frac 34}(|A|),|A|^{\frac 34}p^{\frac 12}\log^{-\frac 12}(|A|)\}.$$
Write,
$$ E(A)=\sum_s |A_s|^2 \ll \sum_{s\in M_1} |A_s|^2 ,$$
where $M_1=\{s:|A_s|\gg k_1:=|A|^{-2}E(A)\}.$  Note that we have the trivial estimate $|M_1|\ll |A|^2k_1^{-1}=|A|^4E^{-1}(|A|).$ 
Now by Lemma \ref{ShSc2} we have,
$$ E(A)=\sum_s |A_s|^2 \ll \frac{E(A)}{|A|\log(A)} \sum_{s\in M_2^c} \frac{|A_s|^2}{|(2A)_s|}+\sum_{s \in M_2}|A_s|^2\ll\sum_{s \in M_2}|A_s|^2,$$
where 
$M_2=\{s:s\in M_1, |(2A)_s|\gg k_2:=|A|^{-1}\log^{-1}(|A|)E(A)\}.$

By Lemma \ref{convolution} we have that $k_2|M_2|\ll |A|^{-\frac 13}|2A|^{\frac 43}|M_2|^{\frac 23}$ yielding $|M_2|\ll|2A|^4|A|^{-1}k_2^{-3}$ as long as $|2A|^2|M_2|\ll\min\{|A|^5,p^3|A|^{-1}\}$.  In order to see that first condition is satisfied, one may note that $|M_2|\ll|M_1|$ combined with our assumptions on the size of energy and sumset.  To show that $|2A|^2|M_2|\ll p^3|A|^{-1}$ we use an exponential sum expansion,
$$|M_2|k_2\ll\sum_{s\in M}|(2A)_s|\ll \frac{1}{p}\sum_m \left|\sum_{x\in 2A}e_p(xm)\right|^2  \left(\sum_{x\in M_2}e_p(xm)\right),$$
followed by applying the bound $\max_{m\neq 0} \left|\sum_{x\in M_2}e_p(xm)\right|\ll p^{\frac 12}|M_2|^{\frac 12}|A|^{-\frac 12}$ to give,
$$|M_2|k_2\ll\max\{p^{-1}|2A|^2|M_2|, p^{\frac 12}|2A||M_2|^{\frac 12}|A|^{-\frac 12}\}.$$
If the first of these two bounds hold then we have $E(A)\ll |A||2A|^2p^{-1}\log(|A|)$. 
We may then assume that $|M_2|\ll p|2A|^2|A|^{-1}k_2^{-2}$ which implies that $|2A|^2|M_2|\ll p |2A|^4|A|\log^2(|A|)E^{-2}(A)\ll p^3|A|^{-1}$ .

Therefore, for $|A|\ll p^{\frac 23}$, we have that $|M_2|\ll|2A|^4|A|^{-1}k_2^{-3}$. Using this fact we may again reduce the number of terms,
$$ E(A)=\sum_s |A_s|^2 \ll k_3^2|M_2|+\sum_{s \in M_3}|A_s|^2\ll\sum_{s \in M_3}|A_s|^2,$$
where 
$M_3=\{s:s\in M_2, |A_s|\gg k_3:=|2A|^{-2}|A|^{-1}\log^{-\frac 32}(|A|)E^2(A)\}.$

Finally, applying Lemma \ref {L3energy} we have,
$$E(A)\ll |A|^4|2A|^2\log^{\frac 32}(|A|)E^{-2}(|A|),$$
as long as $|A|^2|M_3|\ll|2A|^2|M_2|\ll\min\{|A|^5,p^3|A|^{-1}\}.$

\section{$E_{3/2}(A)$: Proof of Lemma \ref{E3/2energy}}
Let $l_i= |A_{z_i}|,z_i\neq 0$ where $l_1\geq l_2\geq\dots$ are arranged in decreasing order. For each $z$ in the coset $aA=\{aa':a'\in A\}, a\in\mathbb Z_p$ note that $|A_z|=|A_a|$. 
By the coset $a_iA$ we will mean the coset on which $l_i=|A_{a_i}|$.
Let $M$ be any $A$-invariant subset of the set $\{z:|A_z| \geq k\}$ and $M_i=\cup_{j=1}^i a_jA\subseteq M$. 
Set $k=|2A|^2|A|^{-3}$.

We have that 
$$l_i |A| i \leq\sum_{j=1}^{i}|A| l_j \leq \sum_{z\in M_i}|A_z|.$$

Now $$ \sum_{z\in M_i}|A_z|= \sum_{z\in M_i} \frac{|A_z|}{|(2A)_z|^{\frac 12}}|(2A)_z|^{\frac 12}\leq
\left(\sum_{z} \frac{|A_z|^2}{|2A_z|}\right)^{\frac 12}\left(\sum_{z\in M_i}{|2A_z|}\right)^{\frac 12}.$$

Therefore, by Lemma \ref{ShSc2} we have that 
$$l_i^2 |A|^2i^2\ll |A|\log(|A|) \sum_{z\in M_i}{|2A_z|},$$
Noting that $|M_i|\ll |A|^2k^{-1}$ we have $|M_i||2A|^2\ll |A|^5.$ Therefore we can apply  Lemma \ref{convolution} to give,
$$l_i^2 |A|^2i^2\ll|2A|^{\frac 43}i^{\frac 23}|A|^{\frac 43}\log{|A|}.$$

Therefore $$l_i\ll  |2A|^{\frac 23}i^{-\frac 23}|A|^{-\frac 13}\log^{\frac 12}{|A|},$$ for $i\ll |A-A||A|^{-1}\leq |A|.$

Now,
\begin{align*} \sum_z |A_z|^{\frac 32}&\ll k^{\frac 12}|A|^2+|A|\sum_{i\ll  |A|}|l_i|^{\frac 32}\\
&\ll k^{\frac 12}|A|^2+|A|^{\frac 12}|2A| \log^{\frac 74}(|A|),\end{align*}
giving the desired result.

\section{Exponential Sum Bound: Proof of Lemma \ref{expA}}
We begin by expanding the sum below and performing a basic substitution,
\begin{align*}
|A| \left|\sum_{x\in A}e_p(\lambda x)\right|^2 &= \sum_{y \in A} \left|\sum_{x\in A}e_p(\lambda yx)\right|^2\\
&= \sum_{x_1,x_2\in A} \sum_{y\in A} e_p( \lambda y (x_1 -x_2)) = \sum_{s}|A_s| \sum_{y \in A} e_p(\lambda ys).\\
\end{align*}

Now we may take absolute values and estimate from above,
\begin{align*}
|A| \Phi_A^2 &\leq \sum_{s}|A_s| \left|\sum_{y \in A} e_p(\lambda ys)\right|.\\
\end{align*}

Applying Holder we have,
$$|A| \Phi_A^2 \ll \left( \sum_{s}|A_s|^{\frac 43}\right)^{\frac 34} \left(\sum_s\left|\sum_{y \in A} e_p(\lambda ys)\right|^4\right)^{\frac 14},$$
which by Plancherel gives,
\begin{equation}\label{expeq}|A| \Phi_A^2 \ll \left( \sum_{s}|A_s|^{\frac 43}\right)^{\frac 34} p^{\frac 14}E^{\frac 14}(A).\end{equation}

Now again applying Holder,
$$\sum_{s}|A_s|^{\frac 43}=\sum_{s}|A_s||A_s|^{\frac 13}\ll \left(\sum_s |A_s|^{\frac 32}\right)^{\frac 23}|A|^{\frac 23},$$
and applying Lemma \ref{E3/2energy},
$$\sum_{s}|A_s|^{\frac 43}\ll |A|^{\frac 23}\left(|A|^{\frac 12}|2A| \log^{\frac 74}(|A|)\right)^{\frac 23}\ll|A||2A|^{\frac 23}\log^{\frac 76}(|A|).$$
Putting this estimate into (\ref{expeq}) gives the stated result.


\begin{thebibliography}{99}









\bibitem{glibichuk1}
A. A. Glibichuk, \emph{Combinational properties of sets of residues modulo a prime and the
Erd\"os-Graham problem}, Mat. Zametki 79, no. 3, (2006), 384-395. Translated in Math.
Notes 79, no. 3, (2006), 356-365.







\bibitem{hk}
D. R. Heath-Brown and S. V. Konyagin,\emph{New bounds for Gauss sums
derived from $k{\rm th}$ powers, and for Heilbronn's exponential
sum}, Q. J. Math. 51 (2000), no. 2, 221-235.


















\bibitem{shp}
I. E. Shparlinski, \emph{On Bounds of Gaussian Sums}, Mat. Zametki, 50 (1991),122–130.

\bibitem{s}
I. D. Shkredov, \emph{Some new inequalities in additive combinatorics}, preprint.

\bibitem{ss1}
T. Schoen and I. D. Shkredov, \emph{On a question of Cochrane and Pinner concerning
multiplicative subgroups}, arXiv:1008.0723v2, May 27, 2011, 1-10.

\bibitem{ss2}
\bysame, \emph{Higher moments of convolutions}, arXiv:1110.2986v1, Oct. 13, 2011, 1-35.

\bibitem{sv} I. D. Shkredov and I. V. Vyugin, \emph{On additive shifts of multiplicative subgroups}, arXiv:1102.1172v1, Feb. 6, 2011, 1-18.





\end{thebibliography}
\end{document}